\documentclass[12pt]{article}
\usepackage[cm]{fullpage}
\usepackage[small]{titlesec}
\usepackage[cm]{fullpage}

\usepackage{amsmath,amscd, float,rotating}
\usepackage{pb-diagram}

\usepackage{latexsym}
\usepackage{amsfonts,color}
\usepackage{amsmath}
\usepackage{amssymb}

\numberwithin{equation}{section}
\newcommand{\qed}{\hfill \ensuremath{\Box}}

\def\XXint#1#2#3{{\setbox0=\hbox{$#1{#2#3}{\int}$}
\vcenter{\hbox{$#2#3$}}\kern-.5\wd0}}

\newcommand{\ve}{\varepsilon}

\newcommand{\ddt}[1]{\frac{\partial #1}{\partial t}}

\begin{document}
\newcounter{remark}
\newcounter{theor}
\setcounter{remark}{0} \setcounter{theor}{1}
\newtheorem{claim}{Claim}
\newtheorem{theorem}{Theorem}[section]
\newtheorem{proposition}{Proposition}[section]
\newtheorem{lemma}{Lemma}[section]
\newtheorem{definition}{Definition}[section]
\newtheorem{conjecture}{Conjecture}[section]
\newtheorem{corollary}{Corollary}[section]
\newenvironment{proof}[1][Proof]{\begin{trivlist}
\item[\hskip \labelsep {\bfseries #1}]}{\end{trivlist}}
\newenvironment{remark}[1][Remark]{\addtocounter{remark}{1} \begin{trivlist}
\item[\hskip \labelsep {\bfseries #1
\thesection.\theremark}]}{\end{trivlist}}
\newenvironment{example}[1][Example]{\addtocounter{remark}{1} \begin{trivlist}
\item[\hskip \labelsep {\bfseries #1
\thesection.\theremark}]}{\end{trivlist}}
~

\begin{center}
{\large \bf
A remark on constant scalar curvature K\"ahler metrics on minimal models
\footnote{Jian, W. is supported in part by China Scholarship Council. Shi, Y. is supported in part by NSFC No.11331001 and the Hwa Ying Foundation at Nanjing University. Song, J. is supported in part by National Science Foundation grant DMS-1711439. }}
\bigskip\bigskip

{Wangjian Jian$^*$, Yalong Shi$^{**}$ and Jian Song$^\dagger$} \\

\bigskip

\end{center}

\begin{abstract}

{\footnotesize In this short note, we prove the existence of constant scalar curvature K\"ahler metrics on compact K\"ahler manifolds with semi-ample canonical bundles. }

\end{abstract}

\section{Introduction}

The existence of constant scalar curvature K\"ahler metrics (cscK) is a fundamental problem in K\"ahler geometry and has been extensively studied. It is proved by Tian \cite{T2} that the existence of  a K\"ahler-Einstein metric on a  Fano manifold without holomorphic vector fields is equivalent to the properness of the Mabuchi K-energy and he further conjectured that such equivalence should also hold for the cscK metrics. In other words, the existence of cscK metrics can be formulated as a variational problem. Recently, this conjecture was proved by Chen-Cheng \cite{CC}. In this paper, we will construct cscK metrics on minimal models in a sufficiently small neighborhood of the canonical classes with an additional assumption on the semi-ampleness of the canonical bundle. Such a construction is related to the minimal model program and in the non-general type case, it can be viewed as the collapsing phenomena  of cscK metrics on the minimal model to the unique twisted K\"ahler-Einstein metric on canonical model constructed by Song-Tian \cite{ST}. Our proof makes use of the existence result of Chen-Cheng \cite{CC} and the criterion for the properness of the Mabuchi K-energy developed by Weinkove \cite{W}, Song-Weinkove \cite{SW} and Li-Shi-Yao \cite{LSY}. The following is our main result. 
%
%
\begin{theorem}\label{main1}
Let $X$ be a compact K\"ahler manifold. If the canonical bundle $K_X$ is semi-ample,  then for any K\"ahler class $[\omega]$ on $X$, there exists $\delta_{X, [\omega]}>0$ such that for any $0<\delta<\delta_{X, [\omega]}$, there exists a unique cscK metric in the K\"ahler class $ [K_X]+\delta[\omega]$.
\end{theorem}

A line bundle $L$ is said to be semi-ample if a sufficiently large power of $L$ is base point free.  Theorem \ref{main1} still holds if $K_X$ is nef and  the numerical dimension of $K_X$ is $n-1$. We believe that in general, the semi-ample assumption for $K_X$  should be replaced by the nef condition. Indeed, the abundance conjecture predicts that nef $K_X$ implies semi-ampleness  for projective manifolds. Any projective manifold with nef canonical bundle is also called a minimal model. Suppose $X$ is a projective minimal model, then the pluricanonical systems of $X$ induce a unique fibration 
$$\Phi: X \rightarrow X_{can}$$
over the unique canonical model $X_{can}$ of $X$ by the finite generation of the canonical ring. The Kodaira dimension of $X$, $\textnormal{Kod}(X)$, is then defined to be the complex dimension of the normal projective variety $X_{can}$. When $X$ is of general type,  i.e. when $\textnormal{Kod}(X)=\dim X$, there exists a unique geometric K\"ahler-Einstein metric $g_{can}$ on $X_{can}$ (\cite{Ts, EGZ, So}) and Theorem \ref{main1} generalizes a result of Arezzo-Pacard \cite{AP} on the cscK metrics on minimal surfaces of general type. 
In \cite{ST}, Song-Tian constructed the unique twisted K\"ahler-Einstein metric $g_{can}$ on $X_{can}$ satisfying
\begin{equation}
Ric(g_{can}) = - g_{can} + g_{WP},
\end{equation}
where $g_{WP}$ is the Weil-Petersson metric induced from the variation of the fibres. We propose the following conjecture relating  $g_{can}$ to the cscK metrics constructed in Theorem \ref{main1}. 
\begin{conjecture} \label{conj1} Let $X$ be an $n$-dimensional K\"ahler manifold with semi-ample (or nef) canonical bundle $K_X$. Then for any K\"ahler class $[\omega]$ and any sequence $\epsilon_j \rightarrow 0$, the cscK metrics $g_j \in [K_X]+ \epsilon_j [\omega]$ converge to the twisted K\"ahler-Einstein metric $g_{can}$ on the canonical model $X_{can}$ of $X$. 
\end{conjecture}

The convergence should be both global in Gromov-Hausdorff topology and local in smooth topology away from the singular fibres of $\Phi: X \rightarrow X_{can}$.
 When the canonical model $X_{can}$ is smooth and the morphism $\Phi: X \rightarrow X_{can}$ has no singular fibre, Conjecture \ref{main1} holds from the result of Fine \cite{F} (Theorem 8.1 and its proof) and the construction of the twisted K\"ahler-Einstein metric on $X_{can}$ \cite{ST}. Theorem \ref{main1} and Conjecture \ref{conj1} are also related to the result of Gross-Wilson \cite{GW} from the perspective of the SYZ conjecture. Theorem \ref{main1} and Conjecture \ref{conj1} can be also interpreted by the slope stability introduced by Ross-Thomas \cite{RT}.  We also propose a related conjecture for the long time solutions of the K\"ahler-Ricci flow. 

\begin{conjecture} \label{conj:KRF}Let $X$ be an $n$-dimensional K\"ahler manifold with nef canonical bundle $K_X$ and positive Kodaira dimension. Then for any initial K\"ahler metric $g_0$, the solution $g(t)$ of the normalized K\"ahler-Ricci flow $$\ddt{g} = - Ric(g) - g, ~ g(0)=g_0$$ converges in Gromov-Hausdorff topology to $g_{can}$ and the scalar curvature $R(t)$ converges to  $-\textnormal{Kod}(X)$ in $C^0(X)$, where $\textnormal{Kod}(X)$ is the Kodaira dimension of $X$.  
\end{conjecture}

A special case of Conjecture \ref{conj:KRF} is recently proved by the first author when $K_X$ is semi-ample and $\Phi: X \rightarrow X_{can}$ has no singular fibres \cite{J}. Conjecture \ref{conj:KRF} will strengthen the result of Song-Tian \cite{ST2} on bounds of the scalar curvature along the K\"ahler-Ricci flow. 

In Theorem \ref{main1}, the manifold $X$ admits a fibration of Calabi-Yau manifolds. The following theorem treats fibrations of K\"ahler manifolds of negative first Chern class. 
\begin{theorem} \label{main2} Let $X$ be an $n$-dimensional K\"ahler manifold. If $K_X$ is ample and $X$ admits a holomorphic fibrtation $\Phi: X \rightarrow Y$ over a Riemann surface $Y$, then for any ample line bundle $L$ over $Y$,  there exists a unique cscK  metric in the K\"ahler class  $[L] + \delta [K_X]$ for any sufficiently small $\delta>0$.

\end{theorem}

 Theorem \ref{main2} generalizes some of the  results of Fine in \cite{F} since a special example in Theorem \ref{main2} is a smooth fibration of smooth high genus curves over another high genus curve. It can be applied to more general fibrations of canonical models of general type with possible singular fibres and low genus base curve. For example, any Fermat surface $x_0^d+x_1^d+x_2^d+x_3^d=0, $ of degree $d\geq 5$ in $\mathbb{P}^3$ has ample canonical bundle, and the projection from one of its straight lines extends to a fibration of the surface onto $\mathbb{P}^1$. Theorem \ref{main2} proves the existence of cscK metrics on these Fermat surfaces far away from their canonical classes.
 
As the K\"ahler classes degenerate to a K\"ahler class on the base $Y$, the cscK metrics in Theorem \ref{main2} should converge to a twisted  constant scalar curvature metric on $Y$ as shown in \cite{F, F2} in the case when the fibration map $\Phi$ is submersion.


\section{Proof }

The well-known Aubin-Yau functionals are given by 

$$I_{\chi}(\varphi)=\int_X\;\varphi \left(\frac {\chi^n}{n!}-\frac {\chi_{\varphi}^n}{n!} \right),  ~J_{\chi}(\varphi)=\int_0^1\int_X\;
{\varphi} \left(\frac {\chi^n}{n!}-\frac {\chi_{t\varphi}^n}{n!} \right) dt,$$
where $\chi_\varphi= \chi+\sqrt{-1}\partial\overline{\partial}\varphi$ and $\varphi \in PSH(X, \chi)=\{ \varphi\in C^\infty(X)~|~ \chi_\varphi>0\}$. In fact, the three functionals $I_\chi$, $J_\chi$ and $I_\chi-J_\chi$ are all equivalent in the sense that they are all positive and they can bounded each other with uniform constants. A modified $J$-functional $\hat J_{\eta, \chi}$ associated to a closed (1,1)-form $\eta$ is defined as below and is related to the $J$-equation proposed by Donaldson in \cite{D} 
$$\hat J_{\eta, \chi}:=\int_0^1\,\int_X\;{\varphi}
(\eta\wedge \chi_{t\varphi}^{n-1}-c\chi_{t\varphi}^n)\frac
{dt}{(n-1)!},$$
where $c=\frac{[\eta]\cdot[\chi]^{n-1}}{[\chi]^n}$.  

The cscK metrics are critical points of Mabuchi K-energy function and a useful observation of Chen \cite{C} relates the Mabuchi K-energy to the Aubin-Yau functionals as follows
\begin{equation}\label{Ken}
K_\chi(\varphi)=\int_X\;\log \frac
{\chi_{\varphi}^n}{\chi^n}\,\frac{\chi_{\varphi}^n}{n!}+\hat J_{-Ric(\chi),
\chi}(\varphi). \end{equation}

The Mabuchi K-energy  is said to be proper as defined in \cite{T2} if there are constants $c_1, c_2>0$ such that 
$$K_\chi(\varphi)\geq c_1 I_\chi(\varphi)-c_2$$
for all $\varphi \in PSH(X,\chi)$. Tian \cite{T2} conjectured that the properness of Mabuchi K-energy implies the existence of cscK metrics and the conjecture is recently proved by Chen-Cheng \cite{CC}.
\begin{theorem}[Theorem 4.1 of \cite{CC}]\label{ChCh}
If the Mabuchi K-energy $K_\chi$ is proper on $PSH(X,\chi)$,  there exists a unique cscK metric  in $[\chi]$.
\end{theorem}

The most quantitive criterion for the properness of Mabuchi K-energy is the cone condition established by Song-Weinkove \cite{SW} as the necessary and sufficient condition for the solvability of the $J$-equation. More precisely, if the canonical bundle $K_X$ is ample, the Mabuchi K-energy $K_\chi$ is proper if there exist K\"ahler forms $\chi'\in [\chi]$ and $\eta \in [K_X]$ such that
\begin{equation}\label{sw}
\left( \frac{n [K_X]\cdot [\chi]^{n-1}}{ [\chi]^n} \chi' - (n-1) \eta \right) \wedge (\chi')^{n-2} >0
\end{equation}
as an $(n-1, n-1)$-form. The condition (\ref{sw}) is further strengthened and generalized by Li-Shi-Yao \cite{LSY} by making use of the $\alpha$-invariant. 
Let us recall Tian's $\alpha$-invariant introduced in \cite{T1}. 
\begin{definition}
Let  $(X, \chi)$ be a K\"ahler manifold. The $\alpha$-invariant is defined as
$$\alpha_X([\chi])= \sup\{\alpha>0~|~\exists C_\alpha>0  \ \text{such that}\int_X e^{-\alpha(\varphi - \sup_X \varphi)} \chi^n \le C_\alpha, \, \forall \varphi\in PSH(X,\chi)\}.$$
\end{definition}
It is obvious that the  $\alpha_X([\chi])$ does not depend on the choice $\chi' \in [\chi]$.

The following properness criterion is essentially contained in the proof to Theorem 1.1 in \cite{LSY} (see  Remark 3.1).
\begin{lemma}\label{criterion}
Let $(X,\chi)$ be a compact K\"ahler manifold of dimension $n$.
If there exist $0\leq \ve<\frac{n+1}{n}\alpha_X([\chi])$, a K\"ahler form $\chi' \in [\chi]$ and a form $\eta\in  [K_X] $ such that $\eta+\ve\chi'>0$ and 
 $$
 \Big( (nc+\ve)\chi'-(n-1)\eta \Big)\wedge \chi'^{n-2}  >0, ~c=\frac{ [K_X]\cdot [\chi]^{n-1}}{[\chi]^n}$$
as an $(n-1,n-1)$ form,
then the Mabuchi $K$-energy $K_\chi$ is proper on $PSH(X, \chi).$
\end{lemma}

\begin{proof} We include a sketch of the proof for the lemma. By Jensen's inequality, there exists $C>0$ such that 
$$\int_X\;\log \left(\frac
{\chi_{\varphi}^n}{\chi^n} \right)\,\frac {\chi_{\varphi}^n}{n!}\geq \ve'(I_{\chi}(\varphi)
-J_{\chi}(\varphi))-C,$$
for any $\ve<\ve'<\frac{n+1}{n}\alpha_X([\chi])$ and $\varphi \in PSH(X, \chi)$.
On the other hand, straightforward calculations show that 
$$\ve(I_\chi-J_\chi)+\hat J_{-Ric(\chi),\chi}=\hat J_{\ve\chi-Ric(\chi), \chi}. $$ Since $\ve\chi-Ric(\chi)$ and $\ve\chi'+\eta$ are in the same cohomology class, we have
$$|\hat J_{\ve\chi-Ric(\chi), \chi}-\hat J_{\eta+\ve\chi', \chi}|\leq C$$ for some uniform constant $C$.
By the main theorem of Song-Weinkove \cite{SW}, our condition in the lemma implies $\hat J_{\eta+\ve\chi', \chi}$ has a critical point, so is bounded from below.  These together imply the properness of $K_\chi$ from the formula (\ref{Ken}) and the definition of properness.

\qed
\end{proof}

We will need the following simple observation for the $\alpha$-invariant to apply Lemma \ref{criterion}. 
\begin{lemma}\label{lower} Suppose the canonical bundle $K_X$ is semi-ample. 
Then for any K\"ahler class $[\omega]$,   we have
$$\alpha_X( [K_X]+\delta[\omega])\geq \alpha_X( [K_X]+[\omega])>0$$
for any $0<\delta<1.$
\end{lemma}

\begin{proof}
Choose any K\"ahler form $\omega' \in [\omega]$. Since $K_X$ is semi-ample, we can find a smooth nonnegative closed $(1,1)$-form $\eta\in [K_X]$ and so  $\eta+ t\omega$ is K\"ahler for all $t>0$. Immediately we have for any $0<\delta <1$
$$PSH(X,\eta+\delta\omega)\subset PSH(X,\eta+\omega).$$
By definition, this implies $\alpha_X([K_X]+\delta[\omega])\geq  \alpha_X([K_X]+[\omega])>0$. 

\qed
\end{proof}

Now we can complete the proof of Theorem \ref{main1}. 
\begin{proof}[Proof of Theorem \ref{main1}] Denote by $\kappa$ the Kodaira dimension of $X$. If $\kappa=0$,  by semi-ampleness assumption, a multiple of $K_X$ is trivial, and Theorem 1.1 immediately follows from Yau's solution \cite{Y}  to the Calabi conjecture.  We will then assume $0<\kappa\leq n$. Since $K_X$ is semi-ample, the pluricanonical system $|mK_X|$ maps $X$ into a projective space $\mathbb{P}^{N_m}$ for some sufficiently large $m$.  We can choose   a smooth closed $(1,1)$-form $\eta$ on $X$ to be the pullback of $m^{-1} \omega_{FS}$  on $\mathbb{P}^{N_m}$. Obviously $\eta \in [K_X] $ is nonnegative and $\eta^l=0$ if and only if $l>\kappa$. Let $\omega$ be a K\"ahler form on $X$ as in the assumption of the theorem. We define
$$\chi_\delta = \eta + \delta \omega, ~~ \ve=\alpha_X([K_X]+[\omega])$$
and
$$c_\delta= \frac{[K_X]\cdot [\chi_\delta]^{n-1}}{[\chi_\delta]^n}.$$
Then we have 
\begin{eqnarray*}
c_\delta &=& \frac{\sum_{i=0}^{\kappa-1}\binom{n-1}{i}([K_X])^{i+1}\cdot (\delta[\omega])^{n-1-i}}{\sum_{i=0}^\kappa \binom{n}{i}([K_X])^i\cdot (\delta[\omega])^{n-i}}\\
&=& \frac{\binom{n-1}{\kappa-1}}{\binom{n}{\kappa}}+O(\delta)=\frac{\kappa}{n}+O(\delta).
\end{eqnarray*}
Straightforward calculations show that
\begin{eqnarray*}
& & (nc_\delta+\ve)\chi_\delta^{n-1}-(n-1)\eta\wedge \chi_\delta^{n-2}\\
&=&(nc_\delta+\ve)(\eta + \delta \omega)^{n-1}-(n-1)\eta\wedge (\eta+\delta\omega)^{n-2}\\
 &=& (nc_\delta+\epsilon)\sum_{i=0}^\kappa \binom{n-1}{i} \eta^i\wedge (\delta\omega)^{n-1-i}-(n-1)\sum_{i=0}^{\kappa-1} \binom{n-2}{i} \eta^{i+1}\wedge (\delta\omega)^{n-2-i}\\
 &=& (nc_\delta+\epsilon)(\delta\omega)^{n-1}+\sum_{i=1}^\kappa A_i \eta^i\wedge (\delta\omega)^{n-1-i}
\end{eqnarray*}
with
$$A_i= (nc_\delta+\epsilon)\binom{n-1}{i}-(n-1)\binom{n-2}{i-1}.$$
Furthermore, 
\begin{eqnarray*}
A_i &=& (\kappa+\ve+O(\delta))\cdot\frac{(n-1)!}{i!(n-1-i)!}-\frac{(n-1)!}{(i-1)!(n-1-i)!}\\
&=& \frac{(n-1)!}{i!(n-1-i)!}\big(\kappa+\epsilon-i+O(\delta)\big)\\
&\geq & \frac{(n-1)!}{i!(n-1-i)!}\big(\epsilon+O(\delta)\big).
\end{eqnarray*}
Each $A_i$ is positive if we choose $\delta$ to be sufficiently small.  The smooth $(n-1, n-1)$ form $\eta^i\wedge (\delta\omega)^{n-1-i}$ is non-negative for all $i =1, ..., \kappa$ and $(nc_\delta+\ve)(\delta\omega)^{n-1}$ is strictly positive. This immediately implies that  
$$(nc_\delta+\ve)\chi_\delta^{n-1}-(n-1)\eta\wedge \chi_\delta^{n-2}>0. $$ By Lemma \ref{lower}, we always have $\ve<\frac{n+1}{n}\alpha_X([\chi])$. The existence part of Theorem \ref{main1} immediately follows from  Theorem \ref{ChCh} and Lemma \ref{criterion}. The general uniqueness of cscK metrics is established in \cite{BB}.
\qed
\end{proof}

The proof of Theorem \ref{main2} is similar to that of Theorem \ref{main1}.
\begin{proof}[Proof of Theorem \ref{main2}]  We let $\beta$ be the pullback of a K\"ahler form in $[L]$ on the Riemann surface $Y$ and $\eta \in [K_X]$ a K\"ahler form on $X$. Then we define 
$$\chi_\delta =\beta+\delta\eta,~  c_\delta=\frac{[K_X]\cdot [\chi_\delta]^{n-1}}{[\chi_\delta]^n}$$ for sufficiently small $\delta>0$. 
Direct computation shows that
$$nc_\delta=\frac{n-1}{\delta}+\frac{K_X^n}{K_X^{n-1}\cdot L}+O(\delta).$$
Then we have
$$\Big( nc_\delta\chi_\delta-(n-1)\eta \Big)\wedge \chi_\delta^{n-2} =\left( nc_\delta \xi+ \left(\delta\frac{K_X^n}{K_X^{n-1}\cdot L}+O(\delta^2)\right)\eta\right)\wedge\chi_\delta^{n-2}>0.
$$
The theorem immediately follows by either applying the condition (\ref{sw}) or by  Lemma \ref{criterion} with $\epsilon=0$.
\qed
\end{proof}


\bigskip

\noindent {\bf{Acknowledgements:}} The work is carried out during the first and second authors' visit at Rutgers University. They would like to  thank  Department of Mathematics of Rutgers University for its hospitality.  The first author would like to thank his advisor Professor Gang Tian for his support and encouragement.

\bigskip

{\noindent \footnotesize $^*$ School of Mathematical Sciences\\
Peking University, Beijing, China 100871\\

{\noindent \footnotesize $^{**}$ Department of Mathematics\\
Nanjing University, Nanjing, China 210093\\

\noindent $\dagger$ Department of Mathematics\\
Rutgers University, Piscataway, NJ 08854\\

\end{document}